\documentclass[11pt]{article}
\usepackage[margin=1in]{geometry}
\newcommand{\eps}{\ensuremath{\varepsilon}}
\usepackage{amssymb}
\usepackage{amsthm}
\usepackage{amstext}
\usepackage{amsfonts}
\usepackage{graphicx}

\newtheorem{theorem}{Theorem}
\newtheorem{lemma}{Lemma}

\newtheorem{prop}{Proposition}
\newtheorem{conjecture}{Conjecture}[section]

\begin{document}

\title{Full rainbow matchings in equivalence relations}

\author{David Munh\'a Correia\thanks{St. Hugh's College, University of Oxford, UK,
e-mail:david.munhacanascorreia@st-hughs.ox.ac.uk} \newline \hspace*{1.5em}
\and Liana Yepremyan\thanks{ Department of Mathematics, Statistics, and Computer Science, University of Illinois at Chicago, Chicago, USA,
London School of Economics, Department of Mathematics, London, UK,
e-mail:lyepre2@uic.edu, l.yepremyan@lse.ac.uk}}





\maketitle

\begin{abstract}
We show that if a multigraph $G$  with  maximum edge-multiplicity of at most $\frac{\sqrt{n}}{\log^2 n}$, is edge-coloured by $n$ colours such that each colour class  is a disjoint union of cliques with at least $2n + o(n)$ vertices, then it has a full rainbow matching, that is, a matching where each colour appears exactly once. This   asymptotically solves a question raised by Clemens, Ehrenm\"uller and Pokrovskiy, and   is related to  problems on algebras of sets studied by Grinblat in [Grinblat 2002].
\end{abstract}
\section{Introduction}

This paper is motivated by  a question  of Grinblat which demonstrates a beautiful interplay between measure theory and  combinatorics. Recall that an algebra \(A\) on a set \(X\) is a family of subsets of this set closed under the operations of union and difference of two subsets. In his book~\cite{grinblat2002} and also later~\cite{grinblat2004,grinblat2015}  Grinblat investigated necessary and sufficient conditions under which the union of at most countably many algebras on $X$ equals to $P(X)$, the power set of $X$. In particular, one of the   questions he studied,  as observed by Nivasch and Omri~\cite{NO}, can be phrased about equivalence relations  as follows.   Let $X$ be a finite set and let A be an equivalence relation on $X$, define the \emph{kernel} of $A$, $ker(A)$ to be the set of all elements in $X$, which have non-trivial equivalence classes. Define $\nu(n)$ to  be the minimal number such that if $A_1, \dots , A_n$ are equivalence relations with $ker(A_i)\geq \nu(n)$ for all $i\in [n]$, then $A_1 \dots,  A_n$ contains a \emph{rainbow matching}, that is,  a set of $2n$ distinct elements $x_1$, $y_1$, . . . ,  $x_n$, $y_n\in X$ with $ x_i\sim y_i \in A_i$ for each $i\in [n]$.

In~\cite{grinblat2002} Grinblat showed  that $ 3n-2\leq  \nu (n) \leq 10n/3 + \sqrt{2n/3} $,  and he asked whether the lower bound is the correct answer for all $n\geq 4$. Nivasch and Omri~\cite{NO} improved Grinblat's upper bound  on $\nu(n)$  to $16n/5 + O(1)$. Finally, Clemens, Ehrenm{\"u}ller,  Pokrovskiy~\cite{CEP} improved the bound on $\nu(n)$ to asymptotically best possible, that is,  $\nu(n) = (3+o(1))n$, using  the graph theoretic language. 

If $A_1, \dots, A_n$ are equivalence relations on a set $X$, let the vertices of an edge-coloured multigraph be the elements of $X$ and, for each $i\in[n]$, let $ {x, y}$ be an edge of colour $i$ if $x\sim_{A_i}y$. Each equivalence relation $A_i$ then corresponds to the colour class $i$ in the multigraph $G$  and, each colour class is a disjoint union of non-trivial cliques.  So, Grinblat's original question can be reformulated as follows: we are given  a multigraph $G$ whose  edges  are coloured with $n$ colours and each subgraph induced by a colour class has at least $3n-2$  vertices and is the disjoint union of non-trivial complete graphs. Is it true that then $G$ contains a \emph{ full rainbow matching}, i.e. a set of $n$ disjoint edges, which all have distinct colours? The authors in~\cite{CEP}  showed that  for sufficiently large $n$, if each colour class has at least $(3+o(1))n$ many vertices then such a rainbow matching exists.  Note that this is asymptotically the best bound, as  if we take a disjoint union of $n-1$ triangles, each edge with multiplicity $n$, one edge per colour,  then there is no rainbow matching of size $n$. If $n = 3$, then  $\nu(3)=9>3n-2$, demonstrated by a $3$-factorization of two disjoint $K_4$'s, as shown by Grinblat~\cite{grinblat2002} and also observed by Nivasch and Omri~\cite{NO}.


In~\cite{CEP} the authors proposed to study Grinblat's original problem when  every pair of distinct elements belongs to at most one equivalence relation (Problem 2, in~\cite{CEP}), that is, to determine $\nu'(n)$ such that  if $A_1, A_2,  \dots, A_n$ are equivalence relations with $ker(A_i) \geq \nu'(n)$ and $A_i\cap A_j\subseteq \{(x,x)| x \in X\}$ then $A_1, A_1, \dots, A_n$ contain a rainbow matching. In the graph theoretic language this is equivalent to finding the minimum $\nu'(n)$ such that every simple graph $G$ whose  edges  are coloured with $n$ colours, each subgraph induced by a colour class has at least $\nu'(n)$  vertices and is the disjoint union of non-trivial complete graphs, contains a full rainbow matching.
 The trivial upper bound is $ \nu'(n)\leq \nu(n)\leq (3+o(1)n$. As for the lower bound, the graph composed of  $n-1$ disjoint stars with $n$ edges, one edge per colour, exhibits that $\nu'(n)>2n-2$~\cite{P}.  In this paper we show that asymptotically $2n$ is the correct answer, in fact, the same result holds for multigraphs with bounded  edge multiplicity. Our main result follows.

\begin{theorem}\label{thm:main}
For every  $\delta > 0$ there exists $n_0$ such that for all $n\geq n_0$ the following holds. If $G$ is a multigraph whose edges are coloured with $n$ colours, such that each colour class is a disjoint union of non-trivial cliques with at least $(2 + \delta)n $ vertices, and the edge-multiplicity of $G$ is at most $\frac{\sqrt{n}}{\log^2 n}$ then $G$ contains a full rainbow matching. 
\end{theorem}

We suspect that improving this upper bound on $\nu'(n)$ to something not of  an asymptotic  form  will be  hard. That is  because our problem is closely related to the Brualdi-Ryser-Stein conjecture on Latin squares and its generalizations.  

\subsection{The Brualdi-Ryser-Stein conjecture}
A Latin square of order $n$ is an $n\times n$ square 
with cells filled using $n$ symbols so that every symbol appears once in each row and once in each column.  \emph{A partial transversal} of size $k$  of a Latin square is a set of $k$ entries in the square which all come from distinct rows and columns, and contain distinct symbols. If $k=n$, where $n$ is the order of the Latin square, the partial transversal is simply called a \emph{transversal}. 

\begin{conjecture}[Brualdi-Ryser-Stein]\label{BRS}
Every Latin square of order $n$ has a partial transversal of size $n-1$ and moreover, if $n$ is odd, it has a transversal.
\end{conjecture}
The history of the conjecture is as follows. In 1967, Ryser~\cite{Ryser} conjectured that the number of transversals in a Latin square of order $n$ has the same parity as $n$, so any Latin square of odd order has a transversal (see also~\cite{best2018did}).  Note that for even $n$ this is not true; for  example, the addition table of  $\mathbb{Z}_{2n}$ is a Latin square with no transversal. Brualdi~\cite{brualdi} conjectured that every Latin square of order $n$ has a partial transversal of size $n-1$ and moreover, if $n$ is odd, it has a transversal. Stein~\cite{S} conjectured that a stronger statement holds and the same outcome should hold even in an $n\times n$ array filled with the numbers $1,2, \dots, n$  such that every number occuring exactly $n$ times. Very recently this was disproved by Pokrovskiy and Sudakov~\cite{PS}. The current best bound on the size of the partial transversal in Brualdi-Ryser-Stein conjecture is $n-O(\log^2n)$ proved by Hatami and Shor~\cite{HS}.  

To every Latin square one can assign an edge-colouring of the complete bipartite graph $K_{n,n}$ by colouring the edge $ij$ by the symbol in the cell $(i, j)$. This is a proper colouring, i.e., one in which any edges which share a vertex have distinct colours. Identifying the cell $(i, j)$ with the edge $ij$, a partial transversal corresponds to a rainbow matching of the same size. So Conjecture~\ref{BRS}  says that any proper edge-colouring of $K_{n,n}$ contains a rainbow matching of size $n-1$, and a perfect rainbow matching, when $n$ is odd. Aharoni and Berger conjectured  the following generalization of  Conjecture~\ref{BRS}.

 \begin{conjecture}\label{conj:aharoniberger}Let $G$ be a bipartite multigraph that is properly edge-coloured
with $n$ colours and has at least $n + 1$ edges of each colour. Then $G$ has a rainbow matching using every colour.
\end{conjecture}

Note that Conjecture~\ref{conj:aharoniberger} would imply that $\nu''(n)\leq 2n+2$ where $\nu''(n)$ can be be defined just as $\nu'(n)$ but restricted to only bipartite graphs (however, Conjecture~\ref{conj:aharoniberger} is much stronger, because it allows any edge multiplicity). Pokrovskiy~\cite{P} showed that Conjecture~\ref{conj:aharoniberger} is asymptotically true, in that the conclusion
holds if there are at least $n + o(n)$ edges of each colour.  Keevash and Yepremyan~\cite{KY1} considered the same
question without the bipartiteness assumption and obtained a result somewhat analogous to
Pokrovskiy's. They showed that any multigraph with edge multiplicities
$o(n)$ that is properly edge-coloured by $n$ colours with at least $n + o(n)$ edges of each
colour contains a rainbow matching of size $n-C$, for some large absolute constant $C>0$. A similar result was also obtained independently by Gao, Ramadurai,
Wanless and Wormald~\cite{GRWW}. They showed that every properly edge-coloured multigraph with edge-multiplicity at most $O(\sqrt{n}/\log^2{n})$  such that each colour class has at least $2n + o(n)$ vertices has a full rainbow matching (note that here the colouring is proper while in our case it is not). Our Theorem~\ref{thm:main} is a generalization of  this result.

\subsection{The Proof overview}  We  further generalize the approach developed in~\cite{GRWW}.  If a multigraph $G$ has a certain structure  we construct a full rainbow matching via some randomized algorithm. This consists of finding an almost full rainbow matching by a sequence of random iterations and then completing it to a full rainbow matching by the greedy algorithm. To show that the last part is possible, the so-called differential equation method is used. Informally saying,  one analyzes the random method to show that the graph left at the very end behaves ``nicely" enough to contain such a matching.  To describe how degrees of vertices, edges and other variables are changing in the left-over graph after each random choice, differential equations are used. The following theorem gives the types of multigraph to which the algorithm can be applied. 

For a multigraph whose edges are coloured, let, for a colour $c$, $n_c$ denote the number of vertices in the colour class of $c$ and $e_c$ denote the number of edges. Let also, $d_v$ denote the number of edges incident to the vertex $v$.

\begin{theorem}\label{thm:auxthm}
For any $0 < \sigma_1 < \sigma_2$, there exists $n_0$ such that for all  $n\geq n_0$ the following is true. Suppose $G$ is a multigraph whose edges are coloured with $n$ colours, such that each colour class is a disjoint union of cliques of order at most three. If  the edge-multiplicity of $G$ is at most $ \frac{\sqrt{n}}{\log^2 n}$, and moreover,
\begin{itemize}
\item $n_c\leq 4n$ and $e_c \geq \sigma_2 n$, for every colour $c$,
    \item $d_v \leq \sigma_1 n$, for every vertex $v$,
\end{itemize}
then $G$ contains a full rainbow matching. 
\end{theorem}

In order to prove our main result, Theorem~\ref{thm:main}, we show that every multigraph we are considering contains a subgraph with the structure given in Theorem~\ref{thm:auxthm}. This  is done via  a careful random sampling.

\section{Some Probability Tools}
In this section we gather some classic probability results which we use throughout our proofs. 

\begin{prop} [\textbf{Chernoff bound}, \cite{DP}]
Let $X$ be a sum of $n$ independent $[0,1]$-valued random variables. Then, for all $t \geq 0$, 
$$\mathbb{P}[|X-\mathbb{E}[X]| > t] \leq 2e^{-\frac{2t^2}{n}}$$ 
\end{prop}

\begin{prop} [\textbf{McDiarmid's Inequality}, \cite{McDiarmid}]
Let $k \in \mathbb{N}$ and $\Pi : [k] \rightarrow [k]$ be a permutation chosen uniformly at random. Let also $h$ be a non-negative real-valued function on the set of permuations of $\{1,\dots , k\}$ and define the random variable $Z := h(\Pi)$ and its median $M$. Suppose that there exist constants $c,r > 0$ such that the following two items occur for any (deterministic) permutation $\pi$:
\begin{itemize}
    \item Swapping two coordinates in the permutation $\pi$ changes $h(\pi)$ by at most $c$.
    \item If $h(\pi) = s$, then there is a set of at most $rs$ coordinates such that $h(\pi') \geq s$ for any other permutation $\pi'$ which agrees with $\pi$ on these coordinates.
\end{itemize}
Then, for all $t \geq 0$,

$$\mathbb{P}[|Z-M| \geq t] \leq 4 \exp \left(-\frac{t^2}{16rc^2(M+t)}\right)$$
\end{prop}

\begin{prop} [\textbf{Azuma's Inequality, \cite{DP}}]
Let $X_1,\dots ,X_n$ be independent random variables which take values in some set $S$ and let $f$ be real-valued on $S^n$ such that there exists a constant $c > 0$ such that for any $\mathbf{x} = (x_1, \dots , x_n) \in S^n$, changing a coordinate of $\mathbf{x}$ deviates $f(\mathbf{x})$ by at most $c$. Then, defining $Y = f(X_1,\dots,X_n)$, we have for every $t > 0$,
$$\mathbb{P}[|Y - \mathbb{E}[Y]| \geq t] \leq 2 \exp \left(-\frac{t^2}{2c^2 n}\right).$$
\end{prop}

\section{Proof of Theorem 1}

\begin{proof}
Note that we only need  to consider the case when the cliques mentioned in the statement are of order at most three. This is because any other case can be reduced to this one by deleting edges in such a way that all monochromatic $K_t$'s with $t \geq 4$ are transformed into a disjoint union of $K_3$'s and $K_2$'s without reducing the number of vertices in a colour class. We can also assume that for every colour $c$, $n_c \leq \lceil (2+\delta)n \rceil + 2$.  Finally, without loss of generality, we may assume $\delta < 2$. For any  colour $c$, let $a_c$, $b_c$ be the number of $K_3$, $K_2$ components, respectively, in the colour class of $c$. Note that $n_c = 3a_c + 2b_c$ and $e_c = 3a_c + b_c$.
\\ 

We construct a random subgraph $H \subseteq G$ in the following manner:
\begin{itemize}
    \item Independently, for each monochromatic $K_3$, either delete two of its edges (transforming it into $K_2$), each pair of edges having probability $\frac{1}{4}$ of being deleted or keep the $K_3$ with probability $\frac{1}{4}$.
\end{itemize}

Define for every vertex $v$, $d^{\textit{tr}}_v, d^{\textit{line}}_v$ to be respectively, the number of edges incident to $v$ which belong to a monochromatic $K_3$, $K_2$. Note that $d_v = d^{\textit{tr}}_v + d^{\textit{line}}_v$ and $Cd_v = \frac{d^{\textit{tr}}_v}{2} + d^{\textit{line}}_v$, where $Cd_v$ is colour degree of the vertex $v$. We will now show that with positive probability, $H$ will satisfy the conditions of Theorem~\ref{thm:auxthm}, thus $H$ will have having a full rainbow matching,  and so will $G$.

For every colour $c$, let $X_c$ denote the number of $K_3$ components in the colour class $c$ which remain unchanged, so that $n_c(H) = 2(a_c+b_c) + X_c$ and $e_c(H) = a_c+b_c+2X_c$. Note that $X_c \sim \text{Bin}(a_c,\frac{1}{4})$. Using the Chernoff bound, we  have that for  any $\eps \geq 0$, 
\begin{itemize}
    \item $\mathbb{P}(e_c(H) \leq \frac{n_c - \eps a_c}{2}) \leq 2e^{-\frac{\eps^2 a_c}{8}}$
\end{itemize}

Similarly, for every vertex $v$, let $Y_v$ denote the number of edges incident to $v$, belonging to a monochromatic $K_3$, which aren't deleted. We note that $d_v(H) = d_v^{\textit{line}} + Y_v$ and $Y_v \sim \text{Bin}(d_v^{\textit{tr}},\frac{1}{2})$. Again by the Chernoff bound, for any $ \eps \geq 0$ we have that
\begin{itemize}
    \item $\mathbb{P}(d_v(H) \geq Cd_v + \frac{\eps d_v^{\textit{tr}}}{2}) \leq 2e^{-\frac{\eps^2 d_v^{\textit{tr}}}{2}}$
\end{itemize}

These two facts allow us to get the following bounds on the structure of the subgraph $H$. Let $\eps_1 = \frac{3\delta}{2(2+\delta)}$ and $\eps_2 = \frac{\delta}{8}$. 

For colours $c$ such that $ a_c \leq \frac{ \eps_1 \lceil (2+\delta)n \rceil }{3}$, note that 
$$e_c(H) \geq a_c + b_c = \frac{n_c - a_c}{2} \geq \frac{\lceil (2+\delta)n \rceil - a_c}{2} \geq \frac{(3-\eps_1)\lceil (2+\delta)n \rceil}{6}$$
For the rest of the colours, since $n$ is large, $a_c \geq  2$ and so, 
$$\frac{(3-\eps_1)\lceil (2+\delta)n \rceil}{6} \leq \frac{3n_c-\eps_1 \lceil (2+\delta)n \rceil}{6} \leq \frac{n_c - \eps_1 \left(a_c - \frac{2}{3} \right)}{2} \leq \frac{n_c - \frac{\eps_1 a_c}{2}}{2},$$
where in the third inequality we used $3a_c \leq n_c \leq \lceil (2+\delta)n \rceil + 2$. Therefore,

$$\mathbb{P} \left(e_c(H) \leq \frac{(3-\eps_1)\lceil (2+\delta)n \rceil}{6}\right) \leq \mathbb{P}\left (e_c(H) \leq \frac{n_c - \frac{a_c \eps_1}{2}}{2}\right) = o\left(\frac{1}{n}\right).$$

Hence, with probability $1-o(1)$, we have that that $e_c(H) \geq \frac{(3-\eps_1)\lceil (2+\delta)n \rceil}{6}$ for every colour.

Similarly, for vertices $v$ with  $\eps_2 n < d_v^{\textit{tr}}$, since $Cd_v \leq n$ and $d_v^{\textit{tr}}\leq 2n$, we have
$$\mathbb{P} \left(d_v(H) \geq (1+\eps_2)n \right) \leq \mathbb{P} \left(d_v(H) \geq Cd_v + \eps_2 d_v^{\textit{tr}} \right) = o \left(\frac{1}{n^2} \right)$$
For the rest of the vertices, note that $d_v(H) \leq (1+\eps_2)n$. Hence, since $v(G) \leq (\lceil (2+\delta)n \rceil+2)n = O(n^2)$, we have that with probability $1-o(1)$, $d_v(H) \leq (1+\eps_2)n$ for every vertex.

To conclude, we have that with probability at least $1-o(1)$, for every colour $c$, by the choice of $\eps_1,\eps_2$,
\begin{itemize}
    \item $e_c(H) \geq \frac{(3-\eps_1)\lceil (2+\delta)n \rceil}{6} \geq (2+\delta)(\frac{3-\eps_1}{6})n = (1+\frac{\delta}{4})n$, for every colour $c$,
    \item $d_v(H) \leq (1+\eps_2)n = (1+\frac{\delta}{8})n$, for every vertex $v$,
\end{itemize}
Then, if $n$ is sufficiently large, with positive probability $H$ will satisfy the conditions of Theorem 2 and contain a full rainbow matching. 
\end{proof}

\section{Informal Treatment of Theorem 2}
The proof of Theorem 2 is technical, here we give a heuristic argument  why Theorem 2 is  true. We will formalize this in the next section.  Make a note of the following  greedy deterministic algorithm which we will use throughout the paper. Take a multigraph which is edge-coloured such that each colour class is a disjoint union of cliques and whose number of vertices is at least four times the number of colours. The greedy way of finding a full rainbow matching consists of picking an edge of each colour and deleting its vertices from the graph at each iteration. Indeed, we can check that when we do so, the number of vertices in each of the colour classes of those colours that aren't yet in the matching decreases by at most four. Given the initial assumption on the multigraph, this will then always produce a full rainbow matching. We now describe the randomized algorithm we use. Take a multigraph G which satisfies the conditions stated in Theorem 2.

Informally, the algorithm goes in the following manner:
we first randomly order the colours in order to put them into chunks $C_1,...,C_{\tau}$ of size $\eps n$ (except maybe the last one), so that $\tau = \lceil \frac{1}{\eps} \rceil$ - we will take $\eps = \eps(n)$ to be of a specific order so that the formalities work out; at each iteration $1 \leq i \leq \tau - 1$ we process the chunk $C_i$, that is, we construct a rainbow matching with the colours in $C_i$ and add it to the rainbow matching we have by the previous iterations in order to get a rainbow matching with the colours in $C_1 \cup \dots \cup C_i$; finally, iteration $\tau$ will consist of greedily finding a rainbow matching with the colours in $C_{\tau}$ and adding it to the previous rainbow matching  with the colours in $C_1 \cup \dots \cup C_{\tau - 1}$, thus constructing a full rainbow matching.
\\
At each iteration $1 \leq i \leq \tau - 1$, we will process chunk $C_i$ in such a way that after finishing iteration $i$, we expect to have:
\begin{itemize}
    \item [1)] For every vertex $v$ that survived and $j > i$, $d_v^{C_j}(i) \leq \eps d(i\eps) \sigma_1 n$, where $d_v^{C_j}(i)$ denotes the number of edges of colours in $C_j$ that are incident to $v$ after finishing iteration $i$.
    \item [2)] For every unprocessed  colour $c$, the number of edges in its colour class graph is $\geq e(i\eps) n$.
\end{itemize}
These functions $d,e$ will be defined in Section 5, but here, we will informally guess  what they should be. We first note that we should be able to have $d(0) = 1$ and $e(0) = \sigma_2$, given our assumptions on $G$ and since we randomly order the colours at the beginning of the algorithm.

Let's describe how we will process chunk $C_{i+1}$ in iteration $1 \leq i+1 \leq \tau - 1$ of the algorithm: 
\begin{itemize}
    \item [1)] First, we pick independently and u.a.r an edge of each colour in $C_{i+1}$. Let's say these are the edges that are chosen and that their vertices are marked. From these, we delete (as well as their vertices) from our graph the non-colliding ones and process them into our rainbow matching. Let's call these edges killed.
    \item [2)] We then zap (that is, delete from the graph) each vertex that survives the previous step, independently with some probability (specific to each vertex), so that overall, every vertex has the same probability $p_i$ (which will be defined later in Section 5) of being marked or zapped (let's call this condemned).
    \item [3)] Finally, we greedily process the rest of the colours that were involved in collisions in 1). In the end, we will have constructed a rainbow matching with the colours in $C_{i+1}$ and we add it to the previous rainbow matching with colours in $C_1 \cup \dots \cup C_i$ given by the previous iterations.
\end{itemize}
We will prove in the next section that the effect of the collisions occurring in 1) will be negligible and so, we will be able to perform 3). \ 

Because of this negligible effect of the collisions, for this informal analysis, we can think of a vertex being deleted as being equivalent to a vertex being condemned. 
\

Note that for a vertex v, we have
$$P(\text{v is chosen in step 1}) \leq \frac{d_v^{C_{i+1}}(i)}{\min_{c \in C_{i+1}} e^c(i)} \leq \eps \sigma_1 \frac{d(i\eps)}{e(i\eps)}$$
where $e^c(i)$ denotes the number of edges in the colour class graph of $c$ after iteration $i$.
Hence, we should be able to take $p_i := \eps \sigma_1 \frac{d(i\eps)}{e(i\eps)}$.

Let's now take a look at what we expect to happen to our parameters during iteration $i+1$. Let $a^c(i),b^c(i)$ denote, respectively, the number of $K_3$'s,$K_2$'s in the colour class graph of $c$ after iteration $i$. We ignore cases where 2 vertices in the same component are deleted, since again, we will show this to be negligible in the next section. We can then see that we should have:

\begin{itemize}
    \item [1)] For every unprocessed  colour $c$, $a^c(i+1) \approx a^c(i)(1-3p_i)$ and $b^c(i+1) \approx b^c(i) + p_i(3a^c(i)-2b^c(i))$
    \item [2)] For every vertex $v$ surviving the iteration and $j > i+1$, $d_v^{C_j}(i+1) \approx d_v^{C_j}(i)(1-p_i)$
\end{itemize}
\

Note that step 2) tells us that we should be able to have $d((i+1)\eps) = d(i\eps)(1-p_i)$. Further, note that for every unprocessed colour, the number of edges in its colour class after iteration $i+1$ is $e^c(i+1) = 3a^c(i+1) + b^c(i+1) \approx (3a^c(i) + b^c(i))(1-2p_i) = e^c(i)(1-2p_i)$. Hence, we should also be able to have $e((i+1)\eps) = e(i\eps)(1-2p_i)$.

Since we will choose $\eps = \eps(n)$ so that it tends to 0 as $n$ tends to $\infty$ we should be able to take the derivatives of $d,e$ at $x = i\eps$ where $0 \leq i \leq \tau - 1$. We then can get, given our choice of $p_i$, that:

\begin{itemize}
    \item $d' = -\sigma_1 \frac{d^2}{e}$
    \item  $e' = -2\sigma_1 d$
\end{itemize}
We can solve this to get $d(x) = 1 - (\frac{\sigma_1}{\sigma_2})x$ and $e(x) = \sigma_2 d(x)^2$. 

We are now in position to see if we expect to indeed be able to greedily process the last chunk $C_{\tau}$. For this, we need to check the values of these functions at $x = (\tau - 1)\eps \in [1-\eps,1]$. Note $d((\tau - 1)\eps) \geq 1 - \frac{\sigma_1}{\sigma_2} > 0$, by assumption. 
\

The number of vertices in the colour class of a colour $c \in C_{\tau}$ after iteration $\tau - 1$ is at least the number of edges, which is at least $e((\tau - 1)\eps)n \geq \sigma_2 (1 - \frac{\sigma_1}{\sigma_2})^2 n$. Since $\eps \rightarrow 0$, if $n$ is large then this will be larger than $4\eps n \geq 4|C_{\tau}|$ and therefore, by the discussion in the beginning of this section, we will be able to greedily find a rainbow matching with the colours in $C_{\tau}$. Concluding, we get a full rainbow matching in $G$.

\section{Proof of Theorem 2}
\begin{proof}
We first give some notation and define in detail the algorithm. Our setup is a graph G with the conditions stated. We will also take $\eps = \eps(n) \in [\frac{1}{2\log \log n}, \frac{1}{\log \log n}]$ and moreover, such that $\eps n \in \mathbb{N}$. Note this exists provided that $n$ is large enough.
\\
Initially, take a random permutation of the colours,  which results in a partition of them into sets (which we shall call "chunks") $C_1,...,C_{\tau}$, where $\tau = \lceil \frac{1}{\eps} \rceil$ and every chunk has size $\eps m$ except for maybe the last chunk (which has size at most $\eps m$). We then start with our iterations. 
\\
Let at the start $M = \emptyset$, $G_0 = G$ and denote the random permutation we performed above as iteration 0. At iteration $1 \leq i \leq \tau - 1$, we ''process'' the chunk $C_i$ in the following way. We look at the graph $G_{i-1}$ that we have, that is, the one that we are left with after iteration $i-1$; For each vertex $v$ in this graph, there is a probability $P_{i-1}(v) \in [0,1]$ such that if we pick independently and u.a.r an edge of each colour in $C_i$, the probability that $v$ is incident to one of these edges is $P_{i-1}(v)$; Given this, define for every vertex $v$ the probability $Q_{i-1} (v)$ to be such that $P_{i-1}(v) + Q_{i-1}(v) (1-P_{i-1}(v)) = p_{i-1}$, which we will later define; If there exists such a $Q_{i-1}(v) \in [0,1]$ for every vertex $v$, we continue - if not, the algorithm breaks; We now randomly assign to each vertex a bit $Z_v \in \{0,1\}$; We do this independently across all vertices and such that $Z_v \sim \text{Ber}(Q_{i-1}(v))$; We then proceed with the following steps: 
\begin{itemize}
    \item [\textit{Step 1}] First, pick independently and u.a.r an edge of each colour in $C_i$. Denote these edges by the \textbf{chosen} ones and say that their vertices are the \textbf{marked} vertices. Note by before that $P_{i-1}(v) = \mathbb{P}(v \text{ is marked})$ for every vertex $v$.
    \\
    
    \item [\textit{Step 2}] From the chosen edges, say that two of them \textbf{collide} if they share a vertex. Delete from $G_{i-1}$, along with their vertices, those chosen edges which don't collide with any other. Say that these are the \textbf{killed} edges and their vertices are also \textbf{killed}. Add the killed edges to $M$ and say that the colours of these edges were processed into $M$.
    \\
    
    \item [\textit{Step 3}] For each vertex $v \in V(G_{i-1})$ that survived Step 2 (i.e, wasn't killed), we look at the value of $Z_v$. If this is $1$, then we \textbf{zap} (i.e, delete from $G_{i-1}$) the vertex $v$. Otherwise, we do nothing.
    \\
    
    \item [\textit{Step 4}] Let $\Phi_i \subseteq C_i$ denote the set of colours which haven't yet been processed into $M$ because their chosen edges collide with others. Greedily find a rainbow matching of these colours and delete it from the graph along with its vertices. Add the rainbow matching to $M$, thus processing the colours in $\Phi_i$. Denote the graph resulting from these 4 steps by $G_i$.
    \\
\end{itemize}

Say that a vertex $v$ was \textbf{condemned} if it was \textbf{zapped} or \textbf{marked}. Note then that $\mathbb{P}(v \text{ condemned}) = \mathbb{P}(v \text{ marked}) + \mathbb{P}(v \text{ zapped} | \text{$v$ not marked})$ $\mathbb{P}(v \text{ not marked})= P_{i-1}(v) + \mathbb{P}(Z_v = 1) (1-P_{i-1}(v)) = P_{i-1}(v) + Q_{i-1} (v) (1-P_{i-1}(v)) = p_{i-1}$. 
\

To finish the algorithm, iteration $\tau$ consists of greedily finding a rainbow matching in $G_{\tau - 1}$ of the colours in $C_{\tau}$ and adding it to $M$, thus processing the colours in $C_{\tau}$. 
\\

Note that if the algorithm is successful, $M$ will be a full rainbow matching. However, there are many things that can break the algorithm. Specifically, each of the iterations $1 \leq i \leq \tau - 1$ doesn't break if and only if we can find the probabilities $Q_{i-1}(v)$ and perform Step 4. Moreover, for the algorithm not to break, we also need to be able to perform iteration $\tau$.
\\

Let's now give some notation which we will need in the analysis of the algorithm: Define $\gamma = \frac{\sigma_1}{\sigma_2} < 1$ and $d(x) = 1-\gamma x \in [1-\gamma,1]$ for $0 \leq x \leq 1$. For each colour $c$, let $e_c(x) = e_c(0)d(x)^2$ where $e_c(0) := \frac{e_c}{n}$ (note that $4 \geq e_c(0) \geq \sigma_2$). Let $e^c(i)$ denote the number of edges of colour $c$ in the graph $G_i$, that is in the graph we have after finishing iteration $i$. Define also $e_i := \max_{c \in C_j, j > i} |e_c(i\eps)n - e^c(i)|$. For a vertex $v \in V(G_i)$, that is, one which survived iterations 1 through $i$, let $d_v^{C_j} (i)$ denote the degree of $v$ in $G_i$ with respect to colours in $C_j$. Let then $d_i := \max_{j > i,v \in G_i} (d_v^{C_j} (i) - \eps \sigma_1 d(i\eps) n)^+$. Note that $2 \eps n \geq d_v$ and so, $d_i \leq (2 + 4 \sigma_1)\eps n$. Take for $0 \leq j \leq \tau - 2$, $p_j := \frac{\eps \gamma}{d(j\eps)} + c_j$, where $c_j = \frac{2d_j d(j\eps) + 2\eps \gamma e_j}{\sigma_2 d(j\eps)^3 n} \geq 0$. Let also $\mu := \frac{\sqrt{n}}{\log ^2 n }$, which by assumption is an upper bound on the edge multiplicity between each pair of vertices.
\\

We now begin with the analysis of the algorithm.
\begin{lemma}

There exist constants $t_j = t_j(\sigma_1, \sigma_2) > 0$, ($2 \leq j \leq 5)$ and some $n(\sigma_1, \sigma_2) \in \mathbb{N}$, such that if $n \geq n(\sigma_1, \sigma_2)$ then for every $0 \leq i \leq \tau - 2$, if we have just finished iteration $i$ with $e_i \leq t_2 n$, then with positive probability we can perform iteration $i+1$ (that is we can successfully process the chunk $C_{i+1}$) and get, at the end, 
\begin{itemize}
    \item $d_{i+1} \leq d_i + t_3 \eps^3 n$
    \item $e_{i+1} \leq e_i(1+\frac{t_4 \eps}{d(i\eps)}) + t_5 \eps^2 n + \frac{16d_i}{\sigma_2}$
\end{itemize}
\end{lemma}
\begin{proof}
See Section 6.
\end{proof}

\begin{lemma}
There exists some $n'(\sigma_1, \sigma_2) \in \mathbb{N}$ such that if $n \geq n'(\sigma_1, \sigma_2)$, with positive probability we have $d_0 \leq \eps^2 n$
\end{lemma}
\begin{proof}
Note that for a vertex $v$ and $1 \leq j \leq \tau$, $d_v^{C_j} (0)$ is determined by the random permutation of the colours performed in the beginning. Note by swapping two colours in that permutation we change $d_v^{C_j} (0)$ by at most $2$. Further, if $d_v^{C_j} (0) = s$, there is a set of at most $s$ colours, which for any other permutation in which these are left unchanged, guarantee that $d_v^{C_j} (0) \geq s$. Then, by McDiarmid's inequality, we have that if $M$ is the median of $d_v^{C_j} (0)$, then for $t \geq 0$,
\begin{equation} \label{ineq:1}
\mathbb{P}(|d_v^{C_j} (0) - M| \geq t) \leq 4e^{-\frac{t^2}{64(M+t)}}
\end{equation}
Note then, by integrating, we have that $\mathbb{E}[|d_v^{C_j} (0) - M|] \leq 4 \int_0^{\infty} \mathbb{P}(|d_v^{C_j} (0) - M| \geq t) dt \leq 4 \int_0^{\infty} e^{-\frac{t^2}{64(M+t)}} dt \leq 4\int_0^M e^{-\frac{t^2}{128M}} dt + 4\int_M^{\infty} e^{-\frac{t}{128}} dt \leq 32 \sqrt{2 \pi M} + 512e^{-\frac{M}{128}} \leq 2 \sqrt{512 \pi M} + 512 \pi$.
Now, note this in particular, gives us that $M - \mathbb{E}[d_v^{C_j} (0)] \leq 2 \sqrt{512 \pi M} + 512 \pi$ and thus, $$(\sqrt{M} - \sqrt{512 \pi})^2 - 1024\pi \leq \mathbb{E}[d_v^{C_j} (0)] $$
Since $\mathbb{E}[d_v^{C_j} (0)] \leq \eps d_v \leq \eps \sigma_1 n$, by assumption, we can see then that if $n$ is large enough, $M \leq \eps \sigma_1 n + O(\sqrt{\eps \sigma_1 n})$. Hence, by (\ref{ineq:1}) with $t = \sqrt{\eps n} \log n$, we get 
$$\mathbb{P}(d_v^{C_j} (0) \geq \sqrt{\eps n} \log n + \eps \sigma_1 n + O(\sqrt{\eps \sigma_1 n})) \leq 4e^{-\frac{\eps n \log^2 n}{O(\eps n)}} = o(n^{-4})$$

Concluding, we get with probability $1-o(m^{-4})$,
$$d_v^{C_j} (0) \leq \eps \sigma_1 n + O(\sqrt{\eps n} \log n)$$
Since we have $O(n^3)$ pairs $v,j$, then, with positive probability, we have $d_0 = O(\sqrt{\eps n} \log n) \leq \eps^2 n$ if $n$ is large enough.
\end{proof}

\begin{lemma}
Take the positive constants $t = t_3$, $s = t_4$ and $r = t_5 + \frac{16(t_3 + 1)}{\sigma_2}$. There exists some $n''(\sigma_1, \sigma_2) \in \mathbb{N}$ such that if $n \geq n''(\sigma_1, \sigma_2)$, then with positive probability, we can perform the algorithm up to iteration $\tau - 1$ (that is, we can successfully process all but the last chunk) and get, $\forall \text{ } 0 \leq i \leq \tau - 2$, 
\begin{itemize}
    \item $d_0 \leq \eps^2 n$
    \item $d_{i+1} \leq d_i + t \eps^3 n$
    \item $e_{i+1} \leq e_i(1+\frac{s \eps}{d(i\eps)}) + r \eps^2 n$
\end{itemize}
\end{lemma}
\begin{proof}
Take $n''(\sigma_1, \sigma_2) \geq n'(\sigma_1, \sigma_2),n(\sigma_1, \sigma_2)$ and moreover, such that we have $\frac{1}{\log \log n''} \leq \frac{t_2(1-\gamma)^{\frac{s}{\gamma}}}{r}$. We prove the lemma inductively on $k \geq -1$. The statement of the lemma is the case $k = \tau - 2$. The base case $k = -1$ follows by Lemma 2. Suppose then that it is true for some $\tau - 2 > k \geq -1$. That is, if $n \geq n''(\sigma_1, \sigma_2)$ then with positive probability, we can perform the algorithm up to iteration $k+1$ (that is, we can successfully process all chunks $C_j$ with $j \leq k+1$) and get, $\forall \text{ } 0 \leq i \leq k$, 
\begin{itemize}
    \item $d_0 \leq \eps^2 n$
    \item $d_{i+1} \leq d_i + t \eps^3 n$
    \item $e_{i+1} \leq e_i(1+\frac{s \eps}{d(i\eps)}) + r \eps^2 n$
\end{itemize}

Note that by iterating the above bounds we get that $e_{k+1} \leq (k+1) r \eps^2 n$ $\prod_{j = 1}^{k} (1+\frac{s\eps}{d(j\eps)}) \leq (k+1) r \eps^2 n \text{ exp } (s\eps \sum_{j=1}^{k} \frac{1}{d(j\eps)})$. By comparing the sum to the integral, this is at most $ (k+1) r\eps^2 n \text{ exp } (s \int_{\eps}^{(k+1)\eps} \frac{dx}{1-\gamma x}) $ which since $ \eps , (k+1)\eps \in (0,1)$, is at most $r\eps n \text{ exp } (\frac{s}{\gamma} \log (\frac{1}{1-\gamma})) = \frac{r \eps n}{(1-\gamma)^{\frac{s}{\gamma}}} \leq t_2 n$ as $\eps \leq \frac{t_2(1-\gamma)^{\frac{s}{\gamma}}}{r}$. Also by iterating, we get $d_{k+1} \leq (t+1) \eps^2 n$.

Then, by Lemma 1, with positive probability (conditional on the previous iterations), we can successfully perform iteration $k+2$ and get
\begin{itemize}
    \item $d_{k+2} \leq d_{k+1} + t \eps^3 n$
    \item $e_{k+2} \leq e_{k+1}(1+\frac{s \eps}{d((k+1)\eps)}) + t_5 \eps^2 n + \frac{16d_{k+1}}{\sigma_2} \leq e_{k+1}(1+\frac{s \eps}{d((k+1)\eps)}) + r \eps^2 n$
\end{itemize}
Thus, case $k+1$ is also true and therefore, the Lemma follows by induction.
\end{proof}

Note that Theorem 2 will follow immediately from this last Lemma 3. Indeed, we get that if $n$ is large enough, with positive probability, the algorithm is successful in processing all chunks $C_j$ with $j \leq \tau - 1$. Further, by the same argument as in the proof of Lemma 3, we can also have that $e_{\tau - 1} \leq \frac{r \eps n}{(1-\gamma)^{\frac{s}{\gamma}}}$. That is, for every colour $c \in C_{\tau}$  
$$e^c(\tau - 1) \geq e_c((\tau - 1)\eps)n - \frac{r\eps n}{(1-\gamma)^{\frac{s}{\gamma}}} \geq \sigma_2 (1-\gamma)^2 n - \frac{r\eps n}{(1-\gamma)^{\frac{s}{\gamma}}}$$
Hence, we can take $n$ large enough so that $e^c(\tau - 1) \geq 4\eps n$ for every colour, implying that the last chunk can be greedily processed. Theorem 2 then follows.
\end{proof}

\section{Proof of Lemma 1}
\begin{proof}
Take $t_2 = \frac{\sigma_2(1-\gamma)^2}{2}$.
Let us recall that the setup is that we just finished iteration $i$ with $e_i \leq t_2 n$. We then work with the graph $G_i$ we have after iteration $i$. Therefore, when we talk about vertices, we always mean a vertex in this graph. We are performing iteration $i+1$.

\begin{lemma}
There exists a constant $s_1 = s_1(\sigma_1, \sigma_2) > 0$ such that for every vertex $v$,

$$s_1 \eps \geq p_i \geq \frac{\eps \sigma_1 d(i\eps) n + d_i}{\sigma_2 d(i\eps)^2 n - e_i} \geq P_i(v)$$
\end{lemma}
\begin{proof}
Note that by a union bound on the events that a specific edge incident to $v$ that has colours in $C_{i+1}$ is chosen in Step 1 and the definitions of $d_i,e_i$, we get

$$P_i(v) \leq \frac{\max_v d_v^{C_{i+1}}(i)}{\min_{c \in C_{i+1}} e^c(i)} \leq \frac{\eps \sigma_1 d(i\eps) n + d_i}{\sigma_2 d(i\eps)^2 n - e_i}$$
Note since $e_i \leq t_2 n \leq \frac{\sigma_2 d(i\eps)^2 n}{2}$ then $\sigma_2 d(i\eps)^2 n - e_i > 0$ and so the last inequality is valid. Further, we also have that
$$\frac{\eps \sigma_1 d(i\eps) n + d_i}{\sigma_2 d(i\eps)^2 n - e_i} - \frac{\eps \gamma}{d(i\eps)} = \frac{d_i d(i\eps) + \eps \gamma e_i}{d(i\eps)(\sigma_2 d(i\eps)^2 n - e_i)} \leq c_i $$ where we are using that $\gamma = \frac{\sigma_1}{\sigma_2}$ and that $e_i \leq \frac{\sigma_2 d(i\eps)^2 n}{2}$. Hence, 
$$p_i = c_i + \frac{\eps \gamma}{d(i\eps)} \geq \frac{\eps \sigma_1 d(i\eps) n + d_i}{\sigma_2 d(i\eps)^2 n - e_i}$$
Finally, note that by recalling that $d_i \leq (2+4\sigma_1)\eps n $ and using that $e_i \leq \frac{\sigma_2 d(i\eps)^2 n}{2}$ and that $d(i\eps) \in [1-\gamma,1]$, 
$$c_i = \frac{2d_i d(i\eps) + 2\eps \gamma e_i}{\sigma_2 d(i\eps)^3 n} \leq (\frac{2 (2+4\sigma_1) + \gamma \sigma_2}{\sigma_2 (1-\gamma)^3})\eps$$
Thus, take $s_1 = \frac{2 (2+4\sigma_1) + \gamma \sigma_2}{\sigma_2 (1-\gamma)^3} + \frac{\gamma}{1-\gamma}$ to get $p_i = c_i + \frac{\eps \gamma}{d(i\eps)} \leq s_1 \eps$.
\end{proof}
Note this Lemma 4 shows that if $n$ is large enough, we can find probabilities $Q_i(v)$ for every vertex $v$. Indeed, by the Lemma 4, we get that $\frac{1}{2} \geq p_i \geq \max_v P_i(v)$ and so, we can always find $Q_i(v)$ uniquely for every vertex $v$. 
\begin{lemma}
There exists a constant $s_2 = s_2(\sigma_1, \sigma_2) > 0$ such that with probability $1-o(1)$, 
$$|\Phi_{i+1}| \leq s_2 \eps^2 n$$
\end{lemma}

\begin{proof}
For every vertex $u$, define $X_u$ to be the number of chosen edges (in Step 1 of iteration $i+1$) incident to $u$. Let also $Y_u := X_u 1_{X_u \geq 2}$. This counts the number of edge collisions that the vertex $u$ creates. Since $|\Phi_{i+1}|$ counts the number of colours whose chosen edge collides with others, we have that $|\Phi_{i+1}| \leq \sum_{u} Y_u$ where the sum is over all vertices $u$. Moreover, note that for every such $u$, by definition, $Y_u \leq X_u(X_u - 1)$ giving that $|\Phi_{i+1}| \leq \sum_{u} X_u(X_u - 1)$. 
\\
Denote the set of edges of colours in the chunk $C_{i+1}$ which are incident to $u$ (after iteration $i$) by $\Gamma_u^{C_{i+1}} (i)$, so that $d_u^{C_{i+1}}(i) = |\Gamma_u^{C_{i+1}} (i)|$. Then, note that for distinct edges $e_1, e_2$ in $\Gamma_u^{C_{i+1}} (i)$, 
$$\mathbb{P} (e_1,e_2 \text{ are both chosen}) \leq \frac{1}{(\min_{c \in C_{i+1}} e^c(i))^2} \leq \frac{1}{t_2^2 n^2}$$ where we are using that $e^c(i) \geq e_c(i\eps)n - e_i \geq \sigma_2 (1-\gamma)^2 n - e_i \geq  \sigma_2 (1-\gamma)^2 n - t_2 n = t_2 n$ by definition of $t_2$. Further, $X_u(X_u-1)$ counts the number of ordered pairs of distinct chosen edges in $\Gamma_u^{C_{i+1}} (i)$. Thus, by a union bound on the events that a specific pair is chosen, we can see that $\mathbb{E} [X_u(X_u-1)] \leq \frac{d_u^{C_{i+1}}(i) ^2}{t_2^2 n^2}$. Using that $d_u^{C_{i+1}} (i) \leq 2\eps n$ (which we have since $|C_{i+1}| \leq \eps n$ and each colour class is a disjoint union of $K_3$'s and $K_2$'s), we can also see that $\mathbb{E} [X_u(X_u-1)] \leq (\frac{2 \eps}{t_2^2 n})d_u^{C_{i+1}}(i)$. Summing up over the vertices we get 
$$\mathbb{E}[|\Phi_{i+1}|] \leq \frac{2 \eps}{t_2^2 n} \sum_{u} d_u^{C_{i+1}}(i) \leq \frac{2 \eps}{t_2^2 n} \times 2 \eps n \times \max_{c \in C_{i+1}} e^c(i) \leq \frac{16}{t_2^2} \eps^2 n$$
the second inequality following by the Handshaking Lemma and the last inequality following since $e^c(i) \leq e_c \leq n_c \leq 4n$.
\\
We now prove concentration. Note that $|\Phi_{i+1}|$ is a function of the $\eps n$ chosen edges in Step 1. Further, changing a chosen edge of a certain colour in $C_{i+1}$ will deviate $|\Phi_{i+1}|$ by at most $3$. Hence, by Azuma's Inequality, we have 

$$\mathbb{P} (||\Phi_{i+1}| - \mathbb{E}[|\Phi_{i+1}|]| \geq \sqrt{\eps n} \log n) \leq \frac{2}{n^{\frac{\log n}{18}}}$$

Since $\sqrt{\eps n} \log n = o(\eps^2 n)$ (as $\eps \asymp \frac{1}{\log \log n}$), we have with probability $1-o(1)$, 
$$|\Phi_{i+1}| \leq s_2 \eps^2 n$$
where $s_2 := \frac{16}{t_2^2} + 1$.
\end{proof}
Let's introduce some notation for the next Lemma. Define for a colour $c$, $e^c(i+\frac{1}{2})$ to be the number of edges of that colour that still remain after Step 3 of iteration $i+1$. Similarly, define $d_v^{C_j} (i+\frac{1}{2})$.

\begin{lemma}
There exists a constant $s_3 = s_3(\sigma_1, \sigma_2) > 0$ such that with probability $1-o(1)$, we have:
\begin{itemize}
    \item $|e^c(i+\frac{1}{2}) - e^c(i)(1-2p_i)| \leq s_3(\eps^2 n + |\Phi_{i+1}|)$ for every unprocessed colour $c$ after Step 3, that is every $c \in \Phi_{i+1} \cup \cup_{j=i+2}^{\tau} C_j$.
    \item $d_v^{C_j} (i+\frac{1}{2}) \leq d_v^{C_j} (i)(1-p_i) + s_3 \eps^3 n$ for every vertex $v$ surviving Steps 1,2 and 3 (of iteration $i+1$) and $j > i+1$.
\end{itemize}
\end{lemma}
\begin{proof}
This is the most technical part of the proof and we include it in the Appendix.
\end{proof}
\begin{lemma}
There exists a constant $s_4 = s_4(\sigma_1, \sigma_2) > 0$ such that with probability $1-o(1)$, we can successfully perform iteration $i+1$ and get 
\begin{itemize}
    \item $|e^c(i+1) - e^c(i)(1-2p_i)| \leq s_4\eps^2 n$ for every $c \in \cup_{j=i+2}^{\tau} C_j$.
    \item $d_v^{C_j} (i+1) \leq d_v^{C_j} (i)(1-p_i) + s_3 \eps^3 n$ for every vertex $v$ that survives iteration $i+1$ (that is, $v \in V(G_{i+1})$) and $j > i+1$.
\end{itemize}
\end{lemma}
\begin{proof}
We first need to show that Step 4 is successfully performed (note by Lemma 4, we have already shown that the algorithm doesn't break initially when we are assigning the probabilities $Q_i(u)$). Indeed, by Lemmas 5 and 6, we have that with probability $1-o(1)$, 
\begin{itemize}
    \item $|e^c(i+\frac{1}{2}) - e^c(i)(1-2p_i)| \leq (s_3+s_3s_2)\eps^2 n$ for every $c \in \Phi_{i+1} \cup \cup_{j=i+2}^{\tau} C_j$.
    \item $d_v^{C_j} (i+\frac{1}{2}) \leq d_v^{C_j} (i)(1-p_i) + s_3 \eps^3 n$ for every vertex $v$ that survives Steps 1,2 and 3 and $j > i+1$.
\end{itemize}
Note that from the first bound, using Lemma 4 (which in particular also shows that for $n$ large $p_i \leq s_1 \eps \leq \frac{1}{2}$) and that $e^c(i) \geq t_2 n$ (which follows by $e_i \leq t_2 n$ as we've seen in the proof of Lemma 5), we can get that
$e^c(i+\frac{1}{2}) \geq e^c(i)(1-2p_i) - (s_3+s_3s_2)\eps^2 n \geq t_2 n(1-2p_i) - (s_3 +s_3s_2)\eps^2 n \geq t_2 n - (2s_1 t_2 + s_3 \eps + s_3 s_2 \eps)\eps n$.
\\
Therefore, since $\eps \rightarrow 0$ when $n \rightarrow \infty$ we have that if $n$ is large enough, then $e^c(i+\frac{1}{2}) \geq 4\eps n \geq 4|\Phi_{i+1}|$ and so, we can greedily process the colours in $\Phi_{i+1}$, that is, Step 4 is successfully performed and so, the iteration is successfully performed.

Further, when we do perform Step 4, we greedily delete $2|\Phi_{i+1}|$ vertices from the graph. Therefore, we have that $d^{C_j}_v (i+1) \leq d^{C_j}_v (i+\frac{1}{2})$ for every vertex $v$ surviving the iteration and $j > i+1$ as well as $|e^c(i+1) - e^c(i+\frac{1}{2})| \leq 4|\Phi_{i+1}| \leq 4s_2 \eps^2 n$ (by Lemma 5) for every colour $c \in \cup_{j=i+2}^{\tau} C_j$. Therefore, Lemma 7 follows by taking $s_4 := 4s_2 + s_3 + s_3s_2$.

\end{proof}

To conclude the proof of Lemma 1, define $t_3 = s_3$, $t_4 = \frac{16}{\sigma_2} - 2 > 0$ and $t_5 = s_4 + 4\gamma^2$.
Note then that if the bounds in Lemma 7 are satisfied, then we have the following:

\begin{itemize}
    \item [i)] Take a vertex $v$ surviving iteration $i+1$ and $j > i+1$. Then, $d_v^{C_j} (i+1) - \eps \sigma_1 d((i+1)\eps)n \leq d_v^{C_j} (i)(1-p_i) + s_3 \eps^3 n - \eps \sigma_1 d(i\eps) n(1-p_i) + \eps \sigma_1 n(\gamma \eps - p_i d(i\eps)) \leq d_i(1-p_i) + s_3\eps^3 n $, where we have used that $p_i \geq \frac{\eps \gamma}{d(i\eps)}$ and $d((i+1)\eps) = d(i\eps) - \gamma \eps$. Therefore, $d_{i+1} \leq d_i + s_3 \eps^3 n$. 
    
    \item [ii)] Take a colour $c \in \cup_{j=i+2}^{\tau} C_j$. Note that by the bound in Lemma 7 and the triangle inequality, $|e^c(i+1) - e_c((i+1)\eps) n| \leq s_4 \eps^2 n + |e^c(i)(1-2p_i) - e_c((i+1)\eps) n|$. Further, this second term is equal to $|(e^c(i)-e_c(i\eps)n)(1-2p_i) - 2p_i e_c(i \eps)n + (e_c(i \eps) - e_c((i+1) \eps))n|$. Since $e_c(i \eps) - e_c((i+1) \eps) = e_c(0)\gamma \eps (2d(i\eps) - \gamma \eps)$,  $p_i = \frac{\eps \gamma}{d(i\eps)} + c_i$ and $e_c(0)n = e_c \leq n_c \leq 4n$ this is at most $|(e^c(i)-e_c(i\eps)n)(1-2p_i) - 2c_i e_c(i\eps) n| + 4 \gamma^2 \eps ^2 n \leq e_i(1-2p_i) + 2c_i e_c(i\eps)n + 4 \gamma^2 \eps ^2 n$. Recalling the definition of $c_i$, note that
    $$2 c_i e_c(i\eps) n = \frac{4e_c(0) d_i}{\sigma_2} + \frac{4 \eps \gamma e_c(0) e_i}{\sigma_2 d(i\eps)} \leq \frac{16 d_i}{\sigma_2} + \frac{16}{\sigma_2} \frac{\eps \gamma e_i}{d(i\eps)}$$
    Therefore, since $p_i \geq \frac{\eps \gamma}{d(i\eps)}$, we have
    $$e_{i+1} \leq e_i(1+(\frac{16}{\sigma_2} - 2)\frac{\eps \gamma}{d(i\eps)}) + (s_4 + 4\gamma^2)\eps^2 n + \frac{16 d_i}{\sigma_2}$$
\end{itemize}
By how $t_3,t_4, t_5$ where defined, this gives us what we need, thus finishing the proof of Lemma 1.
\end{proof}

\section{Concluding Remarks}
In this paper we showed that every multigraph  with  maximum edge-multiplicity  at most $\frac{\sqrt{n}}{\log^2 n}$, edge-coloured by $n$ colours such that each colour class  is a disjoint union of cliques with at least $2n + o(n)$ vertices  has a full rainbow matching.  It would be interesting to know  what is the right  multiplicity bound? For general multigraphs, 
the graph composed of $n$ disjoint  triangles with each edge of multiplicity $n-1$, and all edges of a $i$th triangle being  of colour $i$ shows that  there is no full rainbow  matching. We suspect that our bound on the multipliclity bound on the edges is close to the right answer up to the logarithmic factor.  In particular,  we would like to pose the following problem.

\textbf{Problem 1:} Is there a multigraph $G$  with  maximum edge-multiplicity of at most $\sqrt{n}$, edge-coloured by $n$ colours such that each colour class  has at least $2n+o(n)$ vertices and is a disjoint union of non-trivial cliques,  contains no full rainbow matching?

Another problem is to improve the asymptotic  error term on the number of vertices in a colour class. We note that  our proof  can be modified so that the result holds when the size of each colour class is $2n+n^{1-\alpha}$, for some absolute $\alpha>0$. So we would like to ask for a sub-polynomial improvement.  The question below is natural to ask because of the known lower bound~\cite{HS} on the Brualdi-Ryser-Stein conjecture.

\textbf{Problem 2:} Is  it true that  for some constant $C>0$  every simple graph $G$ edge-coloured by $n$ colours such that each colour class  has at least $2n+C\log^2n$ vertices and is a disjoint union of non-trivial cliques, contains a full rainbow matching?

\bibliographystyle{amsplain}









\section{Acknowledgements}
The first author would like to thank the support provided by the London Mathematical Society and the Mathematical Institute at the University of Oxford through an Undergraduate Research Bursary.   The second author would like to thank Alexey Pokrovskiy for introducing her to Grinblat's original problem on multigraphs and for valuable discussions on the topic.

\bibliographystyle{abbrv}
\bibliography{grinblat}

\providecommand{\bysame}{\leavevmode\hbox to3em{\hrulefill}\thinspace}
\providecommand{\MR}{\relax\ifhmode\unskip\space\fi MR }
\providecommand{\MRhref}[2]{%
  \href{http://www.ams.org/mathscinet-getitem?mr=#1}{#2}
}
\providecommand{\href}[2]{#2}
\begin{thebibliography}{10}

\bibitem{best2018did}
D.~Best and I.~M Wanless, \emph{What did {R}yser conjecture?}, arXiv preprint
  arXiv:1801.02893 (2018).

\bibitem{brualdi}
R.~A Brualdi, H.~J. Ryser, et~al., \emph{Combinatorial matrix theory}, vol.~39,
  Springer, 1991.

\bibitem{CEP}
D.~Clemens, J.~Ehrenm{\"u}ller, and A.~Pokrovskiy, \emph{On sets not belonging
  to algebras and rainbow matchings in graphs}, Journal of Combinatorial
  Theory, Series B \textbf{122} (2017), 109 -- 120.

\bibitem{DP}
D.~P. Dubhashi and A.~Panconesi, \emph{Concentration of {M}easure for the
  {A}nalysis of {R}andomised {A}lgorithms},  (2009).

\bibitem{GRWW}
P.~Gao, R.~Ramadurai, I.~Wanless, and N.~Wormald, \emph{Full rainbow matchings
  in graphs and hypergraphs}, arXiv:1709.02665.

\bibitem{grinblat2002}
L.~S. Grinblat, \emph{Algebras of sets and combinatorics}, vol. 214, 2002.

\bibitem{grinblat2004}
\bysame, \emph{Theorems on sets not belonging to algebras}, Electronic Research
  Announcements of the American Mathematical Society \textbf{10} (2004), no.~6,
  51--57.

\bibitem{grinblat2015}
\bysame, \emph{Families of sets not belonging to algebras and combinatorics of
  finite sets of ultrafilters}, Journal of Inequalities and Applications
  \textbf{1} (2015), 1--19.

\bibitem{HS}
P.~Hatami and P.~W. Shor, \emph{A lower bound for the length of a partial
  transversal in a {L}atin square}, Journal of Combinatorial Theory, Series A
  \textbf{115} (2008), no.~7, 1103--1113.

\bibitem{KY1}
P.~Keevash and L.~Yepremyan, \emph{Rainbow matchings in properly-coloured
  multigraphs}, SIAM journal on Discrete Mathematics \textbf{32} (2018), no.~3,
  1577--1584.

\bibitem{McDiarmid}
C.~McDiarmid, \emph{Concentration for independent permutations}, Combinatorics,
  Probability and Computing \textbf{11} (2002), no.~2, 163--178.

\bibitem{NO}
E.~Nivasch and E.~Omri, \emph{Rainbow matchings and algebras of sets}, Graphs
  and Combinatorics \textbf{33}, no.~2, 473--484.

\bibitem{P}
A.~Pokrovskiy,  (2020), Private communication.

\bibitem{PS}
A.~Pokrovskiy and B.~Sudakov, \emph{A counterexample to {S}tein's
  equi-$n$-square conjecture}, Proceedings of AMS.

\bibitem{Ryser}
H.~J. Ryser, \emph{Neuere probleme der kombinatorik}, Vortr{\"a}ge {\"u}ber
  Kombinatorik, Oberwolfach \textbf{69} (1967), 91.

\bibitem{S}
S.~K. Stein, \emph{Transversals of {L}atin squares and their generalizations},
  Pacific J. Math. \textbf{59} (1975), 567--575.

\end{thebibliography}




\section{Appendix: Proof of Lemma 6}

\begin{proof} Take a colour $c \in \Phi_{i+1} \cup \cup_{j=i+2}^{\tau} C_j$. Let us recall that we are in iteration $i+1$. With that in mind, let $T^{(1)}_c,T^{(2)}_c,L_c$ be respectively, the number of $K_3$,$K_3$,$K_2$ components in the colour class of $c$ with at least $1$,$2$,$1$ condemned vertices. From now on, denote $a^c(i),b^c(i)$ by respectively, the number of $K_3$'s, $K_2$'s in the colour class of colour $c$ after iteration $i$. Similarly as before, define $a^c(i+\frac{1}{2}),b^c(i+\frac{1}{2})$ to be respectively, the number of $K_3$'s,$K_2$'s in the colour class of $c$ after Steps 1,2 and 3.
\

Note that a vertex doesn't survive the first 3 steps (i.e, is deleted in the first 3 steps) if and only if it is killed (in Step 2) or zapped (in Step 3). Moreover, this is equivalent to the vertex being condemned with the exception that it can't be simultaneously marked, not killed and not zapped. However, by the definition of the algorithm, this in particular implies that such a vertex which is marked, not killed and not zapped is incident to a chosen edge from a colour in $\Phi_{i+1}$, that is, a chosen edge which is involved in collisions. Thus, by noting that a monochromatic clique survives if none of its vertices are deleted and that a monochromatic $K_3$ becomes a $K_2$ if exactly one of its vertices is deleted, we can get the following bounds: $|a^c(i+\frac{1}{2}) - (a^c(i) - T^{(1)}_c)| \leq 2|\Phi_{i+1}|$ and $|b^c(i+\frac{1}{2}) - (b^c(i) - L_c + T^{(1)}_c - T^{(2)}_c)| \leq 4|\Phi_{i+1}|$. Hence, since $e^c(i) = 3a^c(i) + b^c(i)$ and similarly for $e^c(i+\frac{1}{2})$, then 
\begin{equation} \label{ineq:5}
\left|e^c(i+\frac{1}{2}) - (e^c(i) - L_c - 2T^{(1)}_c - T^{(2)}_c)\right| \leq 10 |\Phi_{i+1}|    
\end{equation}

\begin{lemma}
Let $u,v$ be two distinct vertices. Then if $n$ is large enough, the probability that they are both condemned is at most $4p_i^2$. 
\end{lemma}
\begin{proof}
Note that by a union bound over the pairs of edges $(e_1,e_2)$ of colours in $C_{i+1}$ where $u \in e_1$ and $v \in e_2$, we have that $\mathbb{P}(u,v \text{ are both marked})$ is bounded above by $ \sum_{e_1,e_2} \mathbb{P}(e_1,e_2 \text{ are both chosen}) \leq \frac{(\max_{w} d_w^{C_{i+1}} (i))^2}{(\min_{c \in C_{i+1}} e^c(i))^2} + \frac{\mu}{\min_{c \in C_{i+1}} e^c(i)} \leq p_i^2 + \frac{\sqrt{n}}{t_2 n \log ^2 n } \leq p_i^2 + (\frac{\eps \gamma}{d(i\eps)})^2 \leq 2p_i^2$ if $n$ is large enough (here we are using Lemma 4, that $e^c(i) \geq t_2 n$ (which follows by $e_i \leq t_2 n$ as seen in the proof of Lemma 5), $d(i\eps) \in [1-\gamma,1]$ and $\eps \asymp \frac{1}{\log \log n}$). 
\

Moreover, by independence of the zapping, 
$\mathbb{P}($one of $u,v$ is condemned and the other is zapped$) \leq p_i(Q_i(u)+Q_i(v))  \leq 2p_i^2$. Now the claim follows by a union bound.
\end{proof}

We are now able to calculate the expectations of $L_c,T^{(1)}_c,T^{(2)}_c$. Indeed, by Lemma 8 and some simple use of the Inclusion-Exclusion principle, we can see that 
\begin{itemize}
    \item [i)] $|\mathbb{E}[L_c] - 2p_i b^c(i)| \leq 4p_i^2 b^c(i) \leq 8s^2_1 \eps^2 n$ 
    \item [ii)] $|\mathbb{E}[T^{(1)}_c] - 3p_i a^c(i)| \leq 16p_i^2 a^c(i) \leq \frac{64}{3}s^2_1 \eps^2 n$
    \item [iii)] $\mathbb{E}[T^{(2)}_c] \leq 12p_i^2 a^c(i) \leq 16s^2_1 \eps^2 n$
\end{itemize}
where we are using that $b^c(i) \leq \frac{n_c}{2} \leq 2n$ and $a^c(i) \leq \frac{n_c}{3} \leq \frac{4}{3}n$. Now we need prove  concentration for these random variables. 
\begin{lemma}
There exists a constant $r_1 = r_1(\sigma_1,\sigma_2) > 0$ such that for every $c \in \Phi_{i+1} \cup \cup_{j=i+2}^{\tau} C_j$, with probability at least $1-o(n^{-2})$ we have
$$|L_c - 2p_i b^c(i)| \leq r_1 \eps^2 n.$$
\end{lemma}
\begin{proof}

First, recall that a vertex $v$ is condemned if it is marked or zapped. However, note that this is equivalent to saying that the vertex $v$ is marked or $Z_v = 1$.
\

Fix a colour $c \in \Phi_{i+1} \cup \cup_{j=i+2}^{\tau} C_j$. Let then $L'_c$ denote the number of $K_2$ components in the colour class of $c$ with a vertex $v$ such that $Z_v = 1$. Let $L''_c$ denote the number of $K_2$ components with a marked vertex and let $L'''_c$ denote the number of $K_2$ components with both a marked vertex and vertex $v$ with $Z_v = 1$. We can then write $L_c = L'_c + L''_c - L'''_c$.
\

Note first that $L'_c$ is sum of $b^c(i) \leq 2n$ independent Bernoulli random variables. Hence, by a Chernoff bound, 
$$\mathbb{P} (|L'_c - \mathbb{E}[L'_c]| > \sqrt{n} \log n) \leq 2\text{ exp}\left(-\frac{2n \log^2 n}{b^c(i)}\right) \leq \frac{2}{n^{\log n}}$$
Moreover, note that 
\begin{equation} \label{ineq:2}
\mathbb{E}[L'_c] \leq 2p_i b^c(i)
\end{equation}
since for every vertex $u$, $\mathbb{P}(Z_u = 1) = Q_i(v) \leq p_i$.
\

Secondly, we have that $L''_c$ is a function of the $\eps n$ chosen edges. Further, changing a chosen edge of a certain colour in $C_{i+1}$ will deviate $L''_c$ by at most $2$. Thus, by Azuma's inequality, for $y > 0$,
$$\mathbb{P}(|L''_c - \mathbb{E}[L''_c]| > \sqrt{\eps n} \log n) \leq \frac{2}{n^{\frac{\log n}{8}}} $$

Finally, we deal with $L'''_c$. First note that since the assignment of the random variables $Z_x$ is independent of the choice of edges in Step 1, we have that for a set $\{u,v\}$ of 2 vertices, 
$$\mathbb{P}(\{u,v\} \text{ has both a marked vertex and a vertex $w$ with $Z_w = 1$} ) \leq 4p_i^2$$
Therefore, 
\begin{equation} \label{ineq:3}
\mathbb{E}[L'''_c] \leq 4p_i^2 b^c(i)   
\end{equation}

Now, take the random set of vertices $Z := \{x: Z_x = 1\}$ and let $S$ be some (deterministic) set of vertices. Since $L'_c$ is $\sigma(Z)$-measurable, let $L'_c(S)$ denote the value $L'_c$ takes when $Z=S$. Then, $\mathbb{E} [L'''_c | Z = S] \leq 2p_i L'_c(S)$ (since each vertex is marked with probability at most $p_i$). Moreover, note that given $\{Z = S\}$, $L'''_c$ is a function of the $\eps n$ chosen edges and it deviates by at most 2 when we change one of these edges. Thus, by Azuma's inequality, 
\begin{equation} \label{ineq:4}
\mathbb{P}(|L'''_c - \mathbb{E}[L'''_c|Z = S]| > \sqrt{\eps n} \log n \text{ } |Z=S) \leq \frac{2}{n^{\frac{\log n}{8}}}
\end{equation}

We will now combine all of these bounds we have gotten so far. Define the events $A := \{|L'_c - \mathbb{E}[L'_c]| \leq \sqrt{n} \log n\}$, $B := \{|L''_c - \mathbb{E}[L''_c]| \leq \sqrt{\eps n} \log n\}$ and $C := \{L'''_c \leq \sqrt{\eps n} \log n + 2p_i (\mathbb{E}[L'_c] + \sqrt{n} \log n)\}$. By the discussion above, we have that $\mathbb{P}(A) \geq 1 - \frac{2}{n^{\log n}}$ and $\mathbb{P}(B) \geq 1 - \frac{2}{n^{\frac{\log n}{8}}}$. Also, note that when $\{Z = S\} \subseteq A$, we have that $\mathbb{E}[L'''_c|Z = S] \leq 2p_i L'_c(S) \leq 2p_i(\mathbb{E}[L'_c] + \sqrt{n} \log n)$. Thus, by our application of Azuma's inequality in (\ref{ineq:4}) above, we have $\mathbb{P}(C | A) \geq 1 - \frac{2}{n^{\frac{\log n}{8}}}$. 
Finally, this gives us that $\mathbb{P}(A \cap B \cap C) \geq  \mathbb{P}(B) - \mathbb{P}(\overline{A \cap C}) = \mathbb{P}(B) - (1 - \mathbb{P}(A) \mathbb{P}(C | A)) = 1-o(n^{-2})$ by the bounds we got.
\

Finally, note that using (\ref{ineq:2}) and (\ref{ineq:3}), we can see that the event $A \cap B \cap C$ implies $|L_c - \mathbb{E}[L_c]| \leq 4p_i^2 b^c(i) + |L_c - \mathbb{E}[L'_c + L''_c]| \leq 4p_i^2 b^c(i) + \sqrt{n} \log n + 2\sqrt{\eps n} \log n + 2p_i(2p_i b^c(i) + \sqrt{n} \log n) \leq 17s_1^2 \eps^2 n$ if $n$ is large enough.
\

Concluding, since we have the bound in i), we get that

$$|L_c - 2p_i b^c(i)| \leq |L_c - \mathbb{E}[L_c]| + 8s_1 \eps^2 n \leq 25s_1^2 \eps^2 n $$
The Lemma follows by then taking $r_1 := 25s_1^2$.

\end{proof}

\begin{lemma}
There exists a constant $r_2 = r_2(\sigma_1,\sigma_2) > 0$ such that for every $c \in \Phi_{i+1} \cup \cup_{j=i+2}^{\tau} C_j$, with probability $1-o(n^{-2})$ we have
$$|T^{(1)}_c - 3p_i a^c(i)| \leq r_2 \eps^2 n$$
\end{lemma}

\begin{proof}
The proof is almost exactly the same as the proof of Lemma 9 and so, we omit it.
\end{proof}

\begin{lemma}
There exists a constant $r_3 = r_3(\sigma_1,\sigma_2) > 0$ such that for every $c \in \Phi_{i+1} \cup \cup_{j=i+2}^{\tau} C_j$, with probability $1-o(n^{-2})$ we have
$$T^{(2)}_c \leq r_3 \eps^2 n$$
\end{lemma}

\begin{proof}
Fix a colour $c \in \Phi_{i+1} \cup \cup_{j=i+2}^{\tau} C_j$. Note that by a similar reasoning as in the proof of Lemma 9, we can see that $$T_c^{(2)} = T_c^{(2)'} + T_c^{(2)''}$$
where $T_c^{(2)'}$ denotes the number of $K_3$ components in the colour class of $c$ with at least 2 distinct vertices $v,u$ such that $Z_v = Z_u = 1$ and $T_c^{(2)''}$ denotes the number of $K_3$ components with at most 1 vertex $v$ with $Z_v = 1$ and at least 2 condemned vertices. Let also, $T_c^{(2)'''}$ denote the number of $K_3$ components with exactly 1 vertex $v$ such that $Z_v = 1$.

Note first that $T_c^{(2)'}$ is sum of $a^c(i) \leq \frac{4}{3} n$ independent Bernoulli random variables and that $\mathbb{E}[T_c^{(2)'}] \leq 3p_i^2 a^c(i) \leq 4s_1^2 \eps^2 n$. Hence, by a Chernoff bound, 
$$\mathbb{P} (T_c^{(2)'} > \sqrt{n} \log n + 4s_1^2 \eps^2 n) \leq 2\text{ exp}(-\frac{2n \log^2 n}{a^c(i)}) \leq \frac{2}{n^{\log n}}$$

Secondly note that $\mathbb{E}[T_c^{(2)'''}] \leq 3p_i a^c(i) \leq 4s_1 \eps n$ and that $T_c^{(2)'''}$ is also a sum of $a^c(i) \leq \frac{4}{3} n$ independent Bernoulli random variables so that we also get
$$\mathbb{P} (T_c^{(2)'''} > \sqrt{n} \log n + 4s_1 \eps n) \leq \frac{2}{n^{\log n}}$$

Finally, take the random set of vertices $Z$ as defined in the proof of Lemma 9. Since $T_c^{(2)'''}$ is $\sigma(Z)$-measurable, let $T_c^{(2)'''}(S)$ denote its value when $Z = S$ for a (deterministic) set of vertices $S$. We note then that 
$$\mathbb{E}[T_c^{(2)''}|Z=S] \leq 2p_i T_c^{(2)'''}(S) + 6p_i^2 a^c(i) \leq 2s_1\eps(T_c^{(2)'''}(S) + 4 s_1 \eps n)$$
where we are using that the probability that a vertex is marked is at most $p_i$ and that, by the proof of Lemma 8, the probability that two distinct vertices are marked is at most $2p_i^2$. Moreover, given $Z=S$, $T_c^{(2)''}$ is a function of the $\eps n$ chosen edges and deviates by at most 2 when we change one of these edges. Thus, by Azuma's inequality,

$$\mathbb{P}(T_c^{(2)''} > \sqrt{\eps n} \log n + \mathbb{E}[T_c^{(2)''}|Z=S]|Z=S) \leq \frac{2}{n^{\frac{\log n}{8}}}$$

Joining these observations in a similar fashion as in the proof of Lemma 9, we get that with probability $1-o(n^{-2})$, $T_c^{(2)} = T_c^{(2)'} + T_c^{(2)''} \leq (\sqrt{n} \log n + 4s_1^2 \eps^2 n) + (\sqrt{\eps n} \log n + 2s_1\eps (\sqrt{n} \log n + 8s_1 \eps n)) \leq 21 s_1^2\eps^2 n$ if $n$ is large. Therefore, the Lemma follows by taking $r_3 := 21 s_1^2$.
\end{proof}

Note then that by (\ref{ineq:5}) before, we get for each colour $c \in \Phi
_{i+1} \cup \cup_{j=i+2}^{\tau} C_j$, with probability $1-o(n^{-2})$,
$$|e^c(i+\frac{1}{2}) - e_c(i)(1-2p_i)| \leq (r_1 + 2r_2 + r_3)\eps^2 n + 10|\Phi'_{i+1}|$$

We now deal with the degrees.

Let's first define for vertices $u,v$, $\mu^{C_j}(k) (uv)$ to be the edge multiplicity of $uv$ with respect to colours in $C_j$ after iteration $k$. Take then a vertex $v$ surviving steps 1,2 and 3 (of iteration $i+1$) and $j > i+1$. We note that $d_v^{C_j} (i+\frac{1}{2}) = d_v^{C_j} (i) - D^{v,j}_1 + D^{v,j}_2$, where 
\begin{itemize}
    \item $D^{v,j}_1 = \sum_{w \in \Gamma^{C_j}_v (i)} \mu^{C_j} (i) (vw) \text{ } 1_{w \text{ is condemned}}$
    \item $D^{v,j}_2 = \sum_{w \in \Gamma^{C_j}_v (i)} \mu^{C_j} (i) (vw) \text{ } 1_{w \text{ is condemned but not deleted}}$
\end{itemize}
We first look at $D^{v,j}_2$. Note that $D^{v,j}_2 \leq D^{v,j}_3$ where 
\begin{itemize}
    \item $D^{v,j}_3 = \sum_{w \in \Gamma^{C_j}_v (i)} \mu^{C_j} (i) (vw) \text{ } 1_{w \text{ is marked but not killed}}$
\end{itemize}
To calculate the expectation of this random variable we require the following simple Lemma:
\begin{lemma}
For a vertex $w$, the probability that it is marked but not killed is at most $2p_i^2$.
\end{lemma}
\begin{proof}
Note that $w$ is marked but not killed if and only if one of the following situations occurs: 
\begin{itemize}
    \item [1)] Two edges incident to $w$ are chosen;
    \item [2)] Two edges $wu,uy$ are chosen, where $y \neq w$.
\end{itemize}
The first situation has probability at most $\frac{d^{C_{i+1}}_w (i)^2}{(\min_{c \in C_{i+1}} e^c(i))^2} \leq p_i^2$ (by Lemma 4) of occurring, while the second situation has probability at most
$$\frac{d^{C_{i+1}}_w (i) \times \max_u d^{C_{i+1}}_u (i)}{(\min_{c \in C_{i+1}} e^c(i))^2} \leq p_i^2$$
which follows also by Lemma 4). Thus, the lemma follows.
\end{proof}

Equipped with Lemma 12, we can now say that $\mathbb{E}[D^{v,j}_3 (i)] \leq 2p_i^2 d_v^{C_j} (i) \leq 4s_1^2 \eps^3 n$, where we are using the trivial bound $d_v^{C_j} (i) \leq 2\eps n$ and Lemma 4. Moreover, note that $D^{v,j}_3$ is a function of the $\eps n$ chosen edges and deviates by at most $4 \mu$ when we change a chosen edge of a certain colour. Thus, by Azuma's inequality we have 
$$\mathbb{P}(|D^{v,j}_3 - \mathbb{E}[D^{v,j}_3]|>\mu \sqrt{\eps n} \log n) \leq \frac{2}{n^{\frac{\log n}{32}}}$$
and so, since $\mu \sqrt{\eps n} \log n = o(\eps^3 n)$, with probability $1-o(n^{-4})$ we have
$$D^{v,j}_2 \leq D^{v,j}_3 \leq (4s_1^2+1) \eps^3 n$$
Finally, we deal with $D^{v,j}_1$. We need the following Lemma.

\begin{lemma}
There exists a constant $r_4 = r_4(\sigma_1,\sigma_2)> 0$ such that for every vertex $v$ surviving steps 1,2 and 3 and $j > i + 1$, with probability $1-o(n^{-4})$ we have
$$|D^{v,j}_1 - p_i d_v^{C_j} (i)| \leq r_4 \eps^3 n$$
\end{lemma}

\begin{proof}
The proof follows exactly the same reasoning as the proofs of Lemmas 9 and 10 and further, exploits the fact that $\mu = \frac{\sqrt{n}}{\log ^2 n}$.
\end{proof}
\

Recall that $d_v^{C_j} (i+\frac{1}{2}) = d_v^{C_j} (i) - D^{v,j}_1 + D^{v,j}_2$. But then, by this and Lemma 13, we can now see that for every vertex $v$ surviving steps 1,2 and 3 and $j > i + 1$, with probability $1-o(n^{-4})$ we have $$d_v^{C_j} (i+\frac{1}{2}) \leq d_v^{C_j}(i)(1-p_i) + (4s_1^2 + r_4 + 1)\eps^3 n$$ 
\

To finish off the proof of Lemma 6, note that we have at most $n$ colours in $\Phi_{i+1} \cup \cup_{j=i+2}^{\tau} C_j$ and at most $\frac{4n^2}{\eps} = O(n^3)$ pairs $(v,j)$ where $\tau \geq j > i +1$ and $v$ is a vertex. We then get, by setting $s_3 := \max\{10,r_1+2r_2+r_3,4s_1^2 + r_4 + 1\}$, that with probability $1-o(1)$, 
\begin{itemize}
    \item $|e^c(i+\frac{1}{2}) - e^c(i)(1-2p_i)| \leq s_3(\eps^2 n + |\Phi_{i+1}|)$ for every colour $c \in \Phi_{i+1} \cup \cup_{j=i+2}^{\tau} C_j$,
    \item $d_v^{C_j} (i+\frac{1}{2}) \leq d_v^{C_j} (i)(1-p_i) + s_3 \eps^3 n$ for every vertex $v$ surviving step 3 and $j > i+1$.
\end{itemize}

\end{proof}

\end{document}